\documentclass[11pt]{amsart}
\usepackage{amsmath,amssymb,graphicx,amsthm,amsfonts,verbatim,mathtools,thmtools}
\usepackage[usenames]{color}
\usepackage[shortlabels]{enumitem}
\usepackage[all,cmtip]{xy}
\newtheorem{thm}{Theorem}
\newtheorem{lem}[thm]{Lemma}
\newtheorem{cor}[thm]{Corollary}
\newtheorem{prop}[thm]{Proposition}
\theoremstyle{definition}
\newtheorem{defn}[thm]{Definition}

\newtheorem{conj}[thm]{Conjecture}

\newtheorem{rmk}[thm]{Remark}

\newcommand{\Cg}{\overline{\mathcal{C}}}

\newcommand{\Me}{\overline{\mathcal{M}}_{1,1}}

\newcommand{\CMe}{\overline{M}_{1,1}}

\newcommand{\pf}{\mathfrak{p}}

\newcommand{\Bc}{\mathcal{B}}

\newcommand{\Oc}{\mathcal{O}}
\newcommand{\Pc}{\mathcal{P}}
\newcommand{\Lc}{\mathcal{L}}
\newcommand{\Nc}{\mathcal{N}}

\newcommand{\Zc}{\mathcal{Z}}

\newcommand{\R}{\mathbb{R}}

\newcommand{\Q}{\mathbb{Q}}
\newcommand{\N}{\mathbb{N}}

\newcommand{\Pb}{\mathbb{P}}
\newcommand{\Ab}{\mathbb{A}}
\newcommand{\A}{\mathbb{A}}

\newcommand{\Fb}{\mathbb{F}}
\newcommand{\Gb}{\mathbb{G}}
\newcommand{\Lb}{\mathbb{L}}

\newcommand{\Qb}{\mathbb{Q}}
\newcommand{\Rb}{\mathbb{R}}
\newcommand{\Zb}{\mathbb{Z}}

\newcommand{\Spec}{Spec}

\DeclareMathOperator{\Hom}{Hom}

\begin{document}
    
    \title[Arithmetic of the moduli of semistable elliptic surfaces]
    {Arithmetic of the moduli of \\ semistable elliptic surfaces}
    
    \author{Changho Han and Jun--Yong Park}
    
    \address{Department of Mathematics, Harvard University, Cambridge, MA 02138}
    \email{chhan@math.harvard.edu}
    
    \address{Center for Geometry and Physics, Institute for Basic Science (IBS), Pohang 37673, Korea}
    \email{junepark@ibs.re.kr}

    \begin{abstract}
        
       We prove a new sharp asymptotic with the lower order term of zeroth order on $\mathcal{Z}_{\mathbb{F}_q(t)}(\mathcal{B})$ for counting the semistable elliptic curves over $\mathbb{F}_q(t)$ by the bounded height of discriminant $\Delta(X)$. The precise count is acquired by considering the moduli of nonsingular semistable elliptic fibrations over $\mathbb{P}^{1}$, also known as semistable elliptic surfaces, with $12n$ nodal singular fibers and a distinguished section. We establish a bijection of $K$-points between the moduli functor of semistable elliptic surfaces and the stack of morphisms $\mathcal{L}_{1,12n} \cong \mathrm{Hom}_n(\mathbb{P}^{1}, \overline{\mathcal{M}}_{1,1})$ where $\overline{\mathcal{M}}_{1,1}$ is the Deligne--Mumford stack of stable elliptic curves and $K$ is any field of characteristic $\neq 2,3$. For $\mathrm{char}(K)=0$, we show that the class of $\mathrm{Hom}_n(\mathbb{P}^1,\mathcal{P}(a,b))$ in the Grothendieck ring of $K$--stacks, where $\mathcal{P}(a,b)$ is a 1-dimensional $(a,b)$ weighted projective stack, is equal to $\mathbb{L}^{(a+b)n+1}-\mathbb{L}^{(a+b)n-1}$. Consequently, we find that the motive of the moduli $\mathcal{L}_{1,12n}$ is $\mathbb{L}^{10n + 1}-\mathbb{L}^{10n - 1}$ and the cardinality of the set of weighted $\mathbb{F}_q$--points to be $\#_q(\mathcal{L}_{1,12n}) =  q^{10n + 1}-q^{10n - 1}$. In the end, we formulate an analogous heuristic on $\mathcal{Z}_{\mathbb{Q}}(\mathcal{B})$ for counting the semistable elliptic curves over $\mathbb{Q}$ by the bounded height of discriminant $\Delta$ through the global fields analogy.
        
    \end{abstract}
    
    \maketitle

    \section{Introduction}\label{sec:intro}

    \par An algebraic surface $X$ is an \textit{elliptic fibration}, if it admits a proper flat morphism $f : X \to \Pb^1$ such that a general fiber is a smooth curve of genus one. $X$ is called an \emph{elliptic surface} in other literatures. It is natural to work with the case when there exists a distinguished section $s: \Pb^1 \hookrightarrow X$ coming from the identity points on each of the elliptic fibers. We restrict our attention to \emph{semistable} elliptic fibrations where all fibers are nodal.
    
    \medskip

    \par Our primary goal of the paper is to enumerate the $\Fb_q$--points of the moduli of nonsingular semistable elliptic surfaces with discriminant degree $12n$. Points on stacks are counted with weights, where a point with its stabilizer group $G$ (e.g. automorphism group of a semistable elliptic surface) contributes a weight $\frac{1}{|G|}$. We acquire the weighted $\Fb_q$--point counts by considering the moduli stack $\Lc_{1, 12n}$ of stable elliptic fibrations over $\Pb^{1}$ with $12n$ nodal singular fibers and a distinguished section. This is justified by showing the equivalence of $K$--points between the two moduli stacks where $K$ is any field of characteristic neither 2 nor 3 (see Proposition~\ref{Semistable_Bijection} for the precise statement and the proof). Regarding $\Me$ as the moduli stack of stable elliptic curves, we show that $\Lc_{1, 12n} \cong \mathrm{Hom}_{n}(\Pb^{1}, \Me)$ a Deligne--Mumford stack parameterizing morphisms from $\Pb^1$ to $\Me$.
    
    \medskip
    
    \par In order to acquire the weighted count of $\Fb_q$--points of the moduli stack $\mathrm{Hom}_{n}(\Pb^{1}, \Me)$, we consider the more general case of $\Hom_n(\Pb^1,\Pc(a,b))$ (see Definition~\ref{def:wtproj}). We provide the explicit stratification of $\Hom_n(\Pb^1,\Pc(a,b))$. In characteristic 0, this allows us to obtain $[\mathrm{Hom}_{n}(\Pb^{1}, \Pc(a,b))]$, a class in the Grothendieck ring of $K$--stacks with $\mathrm{char} (K)$ not dividing $a$ or $b$, expressed as a polynomial of the Lefschetz motive $\Lb:=[\A^1]$. Similarly, we can count $\mathbb{F}_q$-points of $\mathrm{Hom}_{n}(\Pb^{1}, \Pc(a,b))$ up to weights, acquiring the weighted point count $\#_q(\Hom_n(\Pb^1,\Pc(a,b)))=q^{(a+b)n+1}-q^{(a+b)n-1}$ .
    
    \medskip
    
    \begin{thm} [Motive of the moduli stack $\Hom_n(\Pb^1,\Pc(a,b))$]
        \label{thm:motivecount}
        If $\mathrm{char}(K)$ is 0, then the class $[\Hom_n(\Pb^1,\Pc(a,b))]$ in $K_0(\mathrm{Stck}_{K})$ is equivalent to
        \[[\Hom_n(\Pb^1,\Pc(a,b))] = \Lb^{(a+b)n+1}-\Lb^{(a+b)n-1}\;.\]
        On the other hand, when $\mathrm{char}(\mathbb{F}_q)$ does not divide $a$ or $b$, then
        \[\#_q(\Hom_n(\Pb^1,\Pc(a,b)))=q^{(a+b)n+1}-q^{(a+b)n-1}\;.\]
    \end{thm}
    
    \medskip
    
    \par Then, by recognizing $\Me \cong \Pc(4,6)$ over any field $K$ of characteristic $\neq 2,3$ , we conclude the following on the moduli stack of nonsingular semistable elliptic surfaces:
    
    \medskip
    
    \begin{cor} [Motive and weighted point count of the moduli stack $\Lc_{1, 12n}$]
        \label{cor:pointcount}
        If $\text{char}(K) =0$, then
        \[ [\Lc_{1, 12n}]=\Lb^{10n+1}-\Lb^{10n-1}\;.\]
        If $\mathrm{char}(\mathbb{F}_q) \neq 2,3$, \[\#_q(\Lc_{1, 12n})=q^{10n+1}-q^{10n-1}\;.\]
    \end{cor}
    
    \medskip
    
    \par This implies that the number of isomorphism classes of $\Fb_q$-points of $\Lc_{1, 12n}$ is $|\Lc_{1, 12n}(\Fb_q)| = 2 \cdot (q^{10n+1}-q^{10n-1})$ (see Remark~\ref{rmk:ptcount}). Since a semistable elliptic surface $f:X \rightarrow \Pb^1_{\Fb_q}$ is a semistable elliptic curve over $\Pb^1_{\Fb_q}$, we acquire the following count by bounding the \emph{height} of discriminant $\Delta(X)$ when $q$ is not divisible by 6:
    
    \medskip
    
    \begin{thm} [Computation of $\Zc_{\Fb_q(t)}(\Bc)$] 
        The counting of semistable elliptic curves over $\Pb^1_{\Fb_q}$ by $ht(\Delta(X)) = q^{12n} \le \Bc$ satisfies the following inequality:
        \[ \Zc_{\Fb_q(t)}(\Bc) \le 2 \cdot \frac{(q^{11} - q^{9})}{(q^{10}-1)} \cdot \left( \Bc^{\frac{5}{6}} - 1\right)\]
        which is an equality when $\Bc = q^{12n}$ for some $n \in \N$ implying that the acquired upper bound is a sharp asymptotic with the lower order term of zeroth order (i.e. constant).
    \end{thm}
    
    \medskip
    
    \par While the leading term of order $\Oc\left(\Bc^{\frac{5}{6}}\right)$ was expected by the work of \cite{BM}, the sharpness of the upper bound as well as the lower order term of zeroth order over $\Pb^1_{\Fb_q}$ is remarkable as it contrasts the known counting of the stable elliptic curves with squarefree $\Delta$ over $\Qb$ by the work of \cite{Baier} where the error term has the order of $\Oc\left(\Bc^{(7-\frac{5}{27}+\epsilon)/12}\right)$.

    \medskip
    
    \par Lastly, we consider the \textit{global fields analogy}, which says that (global) function fields $\Fb_q(t)$ and algebraic number field $\Q$ are expected to share many properties (see Section~\ref{sec:globfield}). Thus, we formulate the following conjecture by passing the above sharp asymptotic through the global fields analogy:
    
    \medskip
    
    \begin{conj} [Heuristic on $\Zc_{\Qb}(\Bc)$]
        The counting $\Zc_{\Qb}(\Bc)$ of semistable elliptic curves over $\Zb$ by $ht(\Delta) \le \Bc$ follows from the sharp asymptotic counting on $\Zc_{\Fb_q(t)}(\Bc)$ through the global fields analogy. Namely, $\Zc_{\Qb}(\Bc)$ has the leading term of order $\Oc\left(\Bc^{\frac{5}{6}}\right)$ and the lower order term of zeroth order (i.e. constant).
    \end{conj}
    
    \medskip
    
    \par Our project could be considered as an extension of the beautiful work done in \cite{EVW} by Jordan S. Ellenberg, Akshay Venkatesh and Craig Westerland. They proved in loc.cit. a function field analogue of the Cohen-Lenstra heuristics on distributions of class groups by point counting the \textit{Hurwitz spaces} parametrizing branched covers of the complex projective line. As the branched covers of the $\Pb^{1}$ are the fibrations with $0$-dimensional fibers, the moduli of fibrations $f: X \to \Pb^{1}$ on fibered surfaces $X$ with $1$-dimensional fibers is the next most natural case to work on. The counting technique in our project is driven largely by the inspiring work of Benson Farb and Jesse Wolfson \cite{FW} which in turn was motivated by the ideas in Graeme Segal's classical paper \cite{Segal}. 
    
    \medskip
    
    \textbf{Acknowledgments.} Jun-Yong Park would like to express sincere gratitude to his doctoral advisor Craig Westerland for his guidance. We would like to thank Leo Herr and Jonathan Wise for their interest and a lynch pin idea of looking at the weighted projective embeddings. We would also like to thank Denis Auroux, Kenneth Ascher, Dori Bejleri, Jordan S. Ellenberg, Joe Harris, Minhyong Kim, Barry Mazur, Bjorn Poonen, Andr\'as Stipsicz, and Jesse Wolfson for helpful conversations and inspirations. Finally, we thank the referee for in depth review and significant insights throughout the revision guidelines. Jun-Yong Park was supported by IBS-R003-D1, Institute for Basic Science in Korea. Changho Han acknowledge the support of the Natural Sciences and Engineering Research Council of Canada (NSERC), [PGSD3-487436-2016].
    
    \section{Semistable elliptic fibrations over $\Pb^{1}$}\label{sec:fib}
    
    \par In this section, we define the semistable elliptic fibrations over $\Pb^{1}$. For detailed references on elliptic curves and surfaces, we refer the reader to \cite{Silverman, Miranda} respectively.
    
    \smallskip
    
    \par Let us first define the semistable fibrations. Let $X$ be an algebraic surface and $f: X \to \Pb^1$ be a fibration over the projective line with $g > 0$, where $g$ is the genus of a $X_t$ for a general geometric point $t$ of $\Pb^1$. Recall from \S \ref{sec:intro} that a fibration $f$ is equipped with a distinguished section $s:\Pb^1 \hookrightarrow X$.
    
    \begin{defn} 
        A fiber $X_t$ is \textit{semistable}, if it has the following properties:
        \begin{enumerate}
            \item $X_t$ is reduced,
            
            \item The only singularities of $X_t$ are nodes, 
            
            \item $s(t)$ is in the nonsingular locus of $X_t$,
            
            \item $X_t$ contains no $(-1)$-curves of $X$.
            
        \end{enumerate}
        $X_t$ is stable, if in addition $\omega_{X_t}(s(t))$ is ample. The fibration $f$ is (semi)stable, if all the geometric fibers $X_t$ are (semi)stable respectively.
        
    \end{defn}
    
    \par By the semistable reduction theorem \cite[Theorem 8.2]{dJ}, one can always reduce the study of general fibrations to the study of semistable fibrations which are much easier to handle. If $X$ is also nonsingular, then stable fibration can be obtained from $f$ by contracting all $(-2)$-curves (see proof of Proposition~\ref{Semistable_Bijection}). The image of each $(-2)$-curve becomes a singular point on this surface where each singular fiber has only one node.
    
    \medskip
    
    \par In this paper, we work with nonsingular semistable elliptic fibrations where the fiber genus is 1. The only semistable singular fibers with $g(X_t)=1$ are of the type $I_k$ as in \cite[Theorem 6.2]{Kodaira} which are denoted as the type $b_k$ in \cite[Proof of Theorem 1]{Neron}.
    
    \begin{enumerate}
        \item $I_0$ : nonsingular elliptic (generic smooth fiber),
        
        \item $I_1$ : irreducible rational with one node (fishtail singular fiber),
        
        \item $I_{k \ge 2}$ : $k$--cycle of $(-2)$-curves ~ (necklace singular fiber).
        
    \end{enumerate}
    
    \begin{defn} \label{def:ellsurf}
        A nonsingular semistable elliptic surface $X$ is a nonsingular surface equipped with a relatively minimal, semistable elliptic fibration $f : X \to \Pb^1$ that comes with a distinguished section $s: \Pb^1 \hookrightarrow X$ such that the image of $s$ does not intersect nodal singular points of each fiber. We assume that $X$ is not isotrivial, i.e. the trivial elliptic fiber bundles over $\Pb^{1}$ with no singular fibers.
    \end{defn}
    
    \begin{rmk}\label{rmk:disc}
        Any semistable elliptic surface of discriminant degree $12n$ has the $12n$ nodal points distributed over $\mu$ distinct singular fibers of types $I_{k_1},~ \cdots ,~I_{k_i},~ \cdots, I_{k_\mu}$ with $\sum\limits_{i=1}^{\mu} k_i = 12n$. Similarly, any stable elliptic fibration $f:X \rightarrow \Pb^1$ of discriminant degree $12n$ has the $\mu$ distinct singular fibers of type $I_1$ over the $\mu$ distinct points $x_1,~ \cdots ,x_\mu$ on $X$ where each $x_i$ has $A_{k_i-1}$ type singularity such that $\sum\limits_{i=1}^{\mu} k_i = 12n$. Recall that when $\mathrm{char}(K) \neq 2$, an $A_{k}$ surface singularity is \'etale locally defined by \[\overline{K}[x,y,z]/(x^2+y^2+z^{k+1}) \cong \overline{K}[x,y,z]/(xy-z^{k+1}).\] By a convention, $A_0$ means smooth. Moreover, any $A_k$ surface singularity germ $(U,0)$ admits a minimal resolution by \cite{Hironaka} when $\mathrm{char}(K)=0$. When $\mathrm{char}(K)>0$, the minimal resolution exists by \cite{Lipman} as any algebraic surface over a field (which is of finite type by definition) is excellent by \cite[Tag 07QW]{Stacks}. The minimal resolution of $A_k$ singularities can be explicitly computed, which is a sequence of simple blowups
        \[
        U=:U_0 \leftarrow U_1 \leftarrow \dotsb \leftarrow U_k
        \]
        Additionally, the exceptional locus of the minimal resolution $U_k \rightarrow U_0$ consists of the nodal chain of rational curves of length $k$.     
    \end{rmk}
    
    \section{Moduli stack $\Lc_{1, 12n}$ of stable elliptic fibrations over $\Pb^{1}$}
    \label{sec:DM_moduli}
    
    \par In this section, we formulate the moduli stack $\Lc_{1, 12n}$ of stable elliptic fibrations over $\Pb^{1}$ as the Deligne--Mumford stack of morphisms $\mathrm{Hom}_n(\Pb^{1}, \Me)$ and establish the equivalence between the category of semistable elliptic surfaces and that of stable elliptic fibrations over $\Pb^1$. 
    
    \smallskip
    
    \par Let us first recall that a pair $(E,p)$ is a stable elliptic curve if $E$ is a nodal projective curve of arithmetic genus 1 and $p \in E$ is a smooth point. Then, it is well--known by \cite{Knudsen} that $\Me$ is a proper Deligne--Mumford stack of stable elliptic curves with a coarse moduli space $\CMe \cong \Pb^1$ parameterizing the $j$--invariants of elliptic curves. Denote $[\infty] \in \Me$ to be the unique point of $\Me\setminus \mathcal{M}_{1,1}$. Notice that $\Me$ comes equipped with a universal family $p: \Cg_{1,1} \rightarrow \Me$. We consider the following definition for a more concrete description of $\Me$ :
    
    \begin{defn}\label{def:wtproj}
        The 1-dimensional $a,b \in \N$ weighted projective stack is defined as a quotient stack
        \[
        \Pc(a,b):=[(\Ab_{x,y}^2\setminus 0)/\Gb_m]
        \]
        Where $\lambda \in \Gb_m$ acts by $\lambda \cdot (x,y)=(\lambda^a x, \lambda^b y)$. In this case, $x$ and $y$ have degrees $a$ and $b$ respectively. A line bundle $\Oc_{\Pc(a,b)}(m)$ is defined to be a line bundle associated to the sheaf of degree $m$ homogeneous regular functions on $\Ab^2_{x,y}\setminus 0$.
    \end{defn}
    
    \par When the characteristic of the field $K$ is not equal to 2 or 3, \cite{Hassett} shows that $(\Me)_K \cong [ (\Spec~K[a_4,a_6]-(0,0)) / \Gb_m ] = \Pc_K(4,6)$ by using the Weierstrass equations, where $\lambda \cdot a_i=\lambda^i a_i$ for $\lambda \in \Gb_m$ and $i=4,6$. Thus, the $a_i$s have degree $i$ respectively. Note that this is no longer true if characteristic of $K$ is 2 or 3, as the Weierstrass equations are more complicated. Now we can describe the moduli stack of stable elliptic fibrations over $\Pb^1$:
    
    \begin{prop}\label{Lef_Moduli_Space}
        The moduli stack $\Lc_{1, 12n}$ of stable elliptic fibrations over $\Pb^{1}$ with $12n$ nodal singular fibers and a distinguished section is the Deligne--Mumford stack $\mathrm{Hom}_n(\Pb^{1}, \Me)$ parameterizing morphisms $f:\Pb^1 \rightarrow \Me$ such that $f^*\Oc_{\Pc(4,6)}(1) \cong \Oc_{\Pb^1}(n)$.
    \end{prop}
    
    \begin{proof}
        By the definition of the universal family $p$, any stable elliptic fibration $f: Y \rightarrow \Pb^1$ comes from a morphism $\varphi_f:\Pb^1 \rightarrow \Me$ and vice versa. As this correspondence also works in families, we can formulate the moduli of stable elliptic fibrations as $\mathrm{Hom}(\Pb^1,\Me)$. Observe that $\Me \cong \Pc(4,6)$ and its coarse map is $c: \Me \rightarrow \overline{M}_{1,1} \cong \Pb^1$, so that $c$ can be identified with $c: \Pc(4,6) \rightarrow \Pb^1$ where $c(x,y)=[x^3: y^2] \in \Pb^1$ for any $(x,y) \in \Pc(4,6) \cong [(\Ab_{x,y}^2\setminus 0)/\Gb_m]$. Since each coordinate function of $\Pb^1$ lifts to degree 12 functions on $\Pc(4,6)$, we conclude that $c^*\Oc_{\Pb^1}(1) \cong \Oc_{\Pc(4,6)}(12)$. This implies that $\deg (c \circ \varphi_f)=12 \cdot \deg \varphi_f$ where $\deg \varphi_f:=\deg \varphi_f^*\Oc_{\Pc(4,6)}(1)$. Note that the discriminant divisor $\Delta$ of $f$ can be recovered by pulling back $\infty \in \Pb^1$ via $c \circ \varphi_f$.
        
        Above discussion shows that $\Lc_{1,12n} \cong \mathrm{Hom}_n(\Pb^{1},\Me)$. As $\Me$ is Deligne--Mumford, the Hom stack $\mathrm{Hom}(\Pb^{1}, \Me)$ is Deligne--Mumford by \cite{Olsson}. And since $\deg f^*\Oc_{\Pc(4,6)}(1)=n$ is an open condition, $\mathrm{Hom}_n(\Pb^{1}, \Me)$ is an open substack of $\mathrm{Hom}(\Pb^{1}, \Me)$.
    \end{proof}
    
    \begin{rmk}\label{rmk:homstack_DM}
        Analogous proof as above shows that for any $a,b \in \N$ and $\mathrm{char}(K)$ not dividing $a$ or $b$, the stack $\mathrm{Hom}_n(\Pb^1,\Pc(a,b))$ parameterizing morphisms $f:\Pb^1 \rightarrow \Pc(a,b)$ with $f^*\Oc_{\Pc(a,b)}(1) \cong \Oc_{\Pb^1}(n)$ is Deligne--Mumford as well.
    \end{rmk}
    
    \par Above proposition shows that $\Lc_{1,12n}$ is a well-behaving object parametrizing stable elliptic fibrations with discriminant degree $12n$. The proposition below signifies the importance of this stack in regard to understanding the moduli of semistable elliptic surfaces:
    
    \begin{prop}\label{Semistable_Bijection}
        Fix any field $K$ of characteristic $\neq 2,3$. Then there is a canonical equivalence of groupoids between $\Lc_{1,12n}(K)$ and the groupoid of semistable elliptic surfaces over $K$ with discriminant degree $12n$.
    \end{prop}
    
    \par Before we start with the proof, let's recall the facts about the relative minimal model program on surfaces which will be useful. Without the loss of generality, assume that $\mathrm{char}(K)>0$, as all the results below follow analogously when $\mathrm{char}(K)=0$ (see \cite[\S 7--8]{Fujino}).
    
    \par Suppose that we are given a pair $(S,D)$ of a projective normal $\overline{K}$-surface $S$ and an effective $\R$-divisor $D$ on $S$. Then, we have the following extension of log canonical singularities of pairs to arbitrary characteristic:
    
    \begin{defn} \cite[Definition 5.1]{Tanaka}\label{def:lc}
        A pair $(S,D)$ is log canonical (lc for short) if
        \begin{enumerate}
            \item the log canonical divisor $K_S+D$ is $\Rb$-Cartier
            \item for any proper birational morphism $\pi:W \rightarrow S$ and the divisor $D_W$ defined by \[K_W+D_W=\pi^*(K_S+D)\] then $D_W \le 1$, i.e. when writing $D_W=\sum_i a_iE_i$ as a sum of distinct irreducible divisors $E_i$, $a_i \le 1$ for every $i$.
        \end{enumerate}
    \end{defn}
    
    \par For instance, if $S$ is smooth and $D$ is a reduced simple normal crossing divisor, then $(S,D)$ is log canonical.
    
    \smallskip
    
    \par Now consider the relative setting, where we have a projective $\overline{K}$-morphism $f:S \rightarrow C$ into a $\overline{K}$-variety $C$. Assume furthermore that $D$ is a $\Qb$-divisor and $S$ is $\Qb$-factorial. If $(S,D)$ is lc such that $K_S+D$ is not $f$-antinef, then we obtain a $f$-minimal model $f':(S',D') \rightarrow C$ by \cite[Theorem 6.5]{Tanaka}, where $(S',D')$ is lc, $S'$ is $\Qb$-factorial, and $K_{S'}+D'$ is $f'$-nef. Since there is a morphism $\phi:S \rightarrow S'$ with $\phi_*(K_S+D)=K_{S'}+D'$ by loc.cit., $\phi_*$ induces an isomorphism of sheaves from $f_*\Oc(M(K_S+D))$ to $f'_*\Oc(M(K_{S'}+D'))$ for any $M \in \N$.
    
    \par Observe that $K_{S'}+D'$ is $f'$-semiample by the relative log abundance theorem \cite[Theorem 6.9]{Tanaka}. Hence, we obtain the $f$-log canonical model $f'':(S'',D'') \rightarrow C$ as a relative proj over the variety $C$ where $(S'',D'')$ is lc,
    \begin{equation}\label{cano-ring}
    S'':=\underline{\mathrm{Proj}}_C\bigoplus_{M \in \N}f'_*\Oc(M(K_{S'}+D')) \cong \underline{\mathrm{Proj}}_C\bigoplus_{M \in \N}f_*\Oc(M(K_S+D))
    \end{equation}
    and $K_{S''}+D''$ is $f''$-ample. Note that $S''$ is the result of contracting irreducible curves that have 0-intersection with $K_{S'}+D'$. Uniqueness of $(S'',D'')$ from $(S,D)$ follows from the above characterization.
    
    \medskip
    
    \par Now we are ready to tackle the proof of Proposition~\ref{Semistable_Bijection}. References to discussions from Definition~\ref{def:lc} to here are not explicitly specified, unless it is very important to point them out.
    
    \begin{proof}[Proof of Proposition~\ref{Semistable_Bijection}]
        First, we need to construct a functor $\mathcal{F}$ from the groupoid of semistable elliptic surfaces of discriminant degree $12n$ over $K$ to $\Lc_{1,12n}(K)$. Choose any semistable elliptic surface $f : X \to \Pb^1$ with a distinguished section $s: \Pb^1 \hookrightarrow X$ such that the discriminant degree is $12n$. Denote this surface as a triple $(X,f,s)$ and its base change over the algebraic closure $\overline{K}$ of $K$ as $(\overline{X},\overline{f},\overline{s})$. Since $X$ and $s(\Pb^1)$ are smooth, so is the pair $(\overline{X},\overline{s}(\Pb^1_{\overline{K}}))$, hence lc. Observe that $(\overline{X},\overline{s}(\Pb^1_{\overline{K}}))$ is itself a $\overline{f}$-log minimal model as the pair is lc and $\omega_{\overline{X}_t}(\overline{s}(t))$ is nef for any $t \in \Pb^1_{\overline{K}}$. Then there is a $\overline{f}$-log canonical model $\overline{g}:(\overline{Y},\overline{s}'(\Pb^1_{\overline{K}})) \rightarrow \Pb^1_{\overline{K}}$. Since $f: (X,s(\Pb^1)) \rightarrow \Pb^1$ is the fixed locus of the $\mathrm{Gal}(\overline{K}/K)$-action on $\overline{f}:(\overline{X},\overline{s}(\Pb^1_{\overline{K}})) \rightarrow \Pb^1_{\overline{K}}$, there is an induced $\mathrm{Gal}(\overline{K}/K)$-action on $\overline{f}_*\Oc(M(K_{\overline{X}}+\overline{s}(\Pb^1_{\overline{K}})))$. Applying this observation to equation~(\ref{cano-ring}), we can denote $(Y,g,s')$ to be the $\mathrm{Gal}(\overline{K}/K)$-fixed locus of $(\overline{Y},\overline{g},\overline{s}')$. 
        
        To see that $(Y,g,s')$ is a stable elliptic fibration, observe that $Y$ comes from contracting all components of the geometric fibers of $f$ having trivial intersection with the divisor $K_X+s(\Pb^1)$. Since $K_X$ is trivial on each fiber, those components must avoid the distinguished section $s$. Hence, $(Y,g,s')$ is a stable elliptic fibration.
        
        Note that the stable elliptic fibration $(Y,g,s')$ is uniquely determined by $\varphi_g \in \Lc_{1,12m}(K)$ for some $m$ by Proposition~\ref{Lef_Moduli_Space} where $12m$ is the discriminant degree of $Y$. Therefore, we need to show that $\mathcal{F}(X,f,s)=(Y,g,s')$ has the discriminant degree $12n$ (i.e. $m=n$).  
        
        To see that $m=n$, it amounts to finding the configuration of singular fibers of $g$ by Remark~\ref{rmk:disc}. Suppose that a fiber of $f$ at a geometric point $t \in \Pb^1$ is of type $I_k$ with $k>0$, i.e. the fiber $X_t$ consists of a necklace of rational curves of length $k$. Then by contracting the components of $X_t$ not containing $s(t)$, we obtain the fiber $Y_t$ of $g$. Since $k-1$ number of components of $X_t$ are contracted into a point $y_t \in Y_t$, $Y$ is singular of type $A_{k-1}$ at $y_t$ by Remark~\ref{rmk:disc}. This implies that $Y_t$ is \'etale locally cut out by an equation $xy=u^k$ near its unique singular point, where $u$ is an \'etale local parameter at $t \in \Pb^1$. Since \'etale locally the coordinate for the universal family $p$ at the node of the singular fiber at $[\infty] \in \Me$ is $xy = s$ with $s$ a parameter at $[\infty] \in \Me$, $\varphi_g : \Pb^{1} \to \Me$ is ramified at $t \in \Pb^1$ of order $k-1$ via $u^{k} = s$. Hence, $m=n$ so that $\varphi_g \in \Lc_{1,12n}(K)$. Therefore, $\mathcal{F}$ is a functor sending a semistable elliptic surface $(X,f,s)$ to its $f$-log canonical model $\varphi_g$. 
        
        \medskip
        
        Recall that $\mathcal{F}$ is an equivalence iff it is essentially surjective and fully faithful. To see that $\mathcal{F}$ is essentially surjective, choose any $\varphi_h \in \Lc_{1,12n}(K)$. Arguments similar to above show that the corresponding surface $(Z,h,v)$ only has the singularities of type $A_k$ appearing over ramification points of $\varphi_h$ over $[\infty] \in \Me$. Note that an $A_k$ singularity of $Z$ corresponds to a singular point of a fiber of $h$, and its ramification order  with respect to $\varphi_h$ is $k$. Hence, $Z$ has a minimal resolution of singularities $\eta: S \rightarrow Z$ (see Remark~\ref{rmk:disc}) inducing a fibration $q:S \rightarrow \Pb^1$ via $h$. For each geometric point $t \in \Pb^1$, the fiber of $q$ at $t$ is of type $I_k$ whenever $Z$ has an $A_{k-1}$-singular point at the singular point of the fiber 
        $Z_t$. Since the image of $v$ avoids singular points of $Z$, it lifts to $v':\Pb^1 \rightarrow S$ avoiding singular points of fibers of $q$. Therefore, $(S,q,v')$ is a semistable elliptic surface with its log canonical model $(Z,h,v)$ (by analogous arguments in the first paragraph of the proof). To see that the discriminant degree of $(S,q,v')$ is $12n$, observe that the fiber of $c \circ \varphi_h$ at $c([\infty]) \in \Pb^1$ is a collection of points $x_1,\dotsc,x_{\mu}$ with multiplicities $k_1,\dotsc,k_{\mu}$ respectively. These points have ramification orders $k_1-1,\dotsc,k_\mu-1$ respectively. Since $\sum k_i=12n$ is the degree of $c \circ \varphi_h$, above construction induces singular fibers of $q$ exactly at $x_i$'s with fiber type $I_{k_i}$'s. Summing over the singular points of each fiber, $(S,q,v')$ indeed has the discriminant degree equal to $12n$. This shows that $\mathcal{F}(S,q,v') \cong \varphi_h$. 
        
        Finally, $\mathcal{F}$ is full, because any isomorphism $H: \varphi_{h_1} \rightarrow \varphi_{h_2}$ lifts to their minimal resolution $\tilde H:(S_1,q_1,v_1') \rightarrow (S_2,q_2,v_2')$. $\mathcal{F}$ is faithful as any two isomorphisms between nonsingular surfaces $(X_i,f_i,s_i)$'s for $i=1,2$ agreeing on their log canonical models agree on an open subset, hence agreeing everywhere by separatedness of $X_i$'s. This proves that $\mathcal{F}$ is an equivalence.
    \end{proof}
    
    \begin{rmk}
        In fact, it is unclear whether the functor $\mathcal{F}$ in the proof above extends to families over arbitrary $K$-scheme $B$. Since the relative log abundance is a conjecture for sufficiently high dimensions, it is unknown whether the log canonical model can be taken in families. 
        
        If we instead assume that the log abundance conjecture holds, then the functor $\mathcal{F}$ extends, giving a map from the moduli functor of semistable elliptic surfaces and the stack $\Lc_{1,12n}$. However, it is still unclear whether $\mathcal{F}$ is essentially surjective, as a simultaneous minimal resolution of the fibers of families over any base $B$ may not exist (normally, a resolution of singularities create exceptional divisors, something that is not desired for the purpose of $\mathcal{F}$).
    \end{rmk}
    
    \par A very important consequence of Proposition~\ref{Semistable_Bijection} is that the weighted point count of $\Lc_{1,12n}$ gives the same number as that of the moduli of semistable elliptic surfaces. Since the former has a concrete description as a Deligne--Mumford stack, we focus on acquiring the arithmetic invariants of $\Lc_{1,12n}$.
    
    \section{Motive/Point count of $\Hom_n(\Pb^1,\Pc(a,b))$ over finite fields}
    \label{sec:count}
    \par In this section, we enumerate the stack $\Hom_n(\Pb^1,\Pc(a,b))$ over finite fields $\Fb_q$ for $q$ prime power with characteristic not dividing $a$ or $b$ by using the Grothendieck ring of stacks and explicit point counts. This is applied to the case $\Me \cong \Pc(4,6)$ to obtain Corollary~\ref{cor:pointcount}. Fix $n>0$.
    
    \smallskip
    
    \par To perform a weighted point count on $\Hom_n(\Pb^1,\Pc(a,b))$, we use the idea of cut-and-paste by Grothendieck:
    
    \begin{defn}\label{defn:Grothringvar}
        \cite[\S 1]{Ekedahl}
        Fix a field $K$. Then the \emph{Grothendieck ring $K_0(\mathrm{Stck}_K)$ of algebraic stacks of finite type over $K$ all of whose stabilizer group schemes are affine}, is a group generated by isomorphism classes of $K$-stacks $[\mathcal X]$ of finite type, modulo relations:
        \begin{itemize}
            \item $[\mathcal X]=[\mathcal Z]+[\mathcal{X}\setminus \mathcal{Z}]$ for $\mathcal Z \subset \mathcal X$ a closed substack,
            \item $[\mathcal E]=[\mathcal{X} \times \Ab^n ]$ for $\mathcal{E}$ a vector bundle of rank $n$ on $\mathcal X$.
        \end{itemize}
        Multiplication on $K_0(\mathrm{Stck}_K)$ is induced by $[\mathcal X][\mathcal Y]:=[\mathcal X\times_K \mathcal Y]$. There is a distinguished element $\Lb:=[\A^1] \in K_0(\mathrm{Stck}_K)$, called the \emph{Lefschetz motive}.
    \end{defn}
    
    \par Note that all stabilizer group schemes of a stack $\mathcal X$ being affine is equivalent to the diagonal morphism $\mathcal X \rightarrow \mathcal X \times_K \mathcal X$ being affine. Since many algebraic stacks can be written locally as a quotient of a scheme by an algebraic group $\Gb_m$, the following lemma is very useful:
    
    \begin{lem}\label{lem:Gm_quot}
        For any $\Gb_m$-torsor $\mathcal X \rightarrow \mathcal Y$ of finite type algebraic stacks, we have $[\mathcal Y]=[\mathcal X][\Gb_m]^{-1}$.
    \end{lem}
    \begin{proof}
        This follows from \cite[Proposition 1.1 iii), 1.4]{Ekedahl} and the definition of $K_0^{\mathrm{Zar}}(\mathrm{Stck}_K)$ in \cite[\S 1]{Ekedahl}.
    \end{proof}
    
    \par Since any finite type algebraic $\Fb_q$-stack $\mathcal X$ admits a smooth cover $Y \rightarrow \mathcal X$ by a $\Fb_q$-scheme of finite type, the set $|\mathcal X(\Fb_q)|$ of isomorphism classes of $\Fb_q$-points is finite as $|Y(\Fb_q)|$ is finite as well. Hence, we can define:
    
    \begin{defn}\label{def:wtcount}
        The weighted point count of $\mathcal X$ is defined as a sum:
        \[
        \#_q(\mathcal X):=\sum_{x \in |\mathcal X(\Fb_q)|}\frac{1}{|\mathrm{Stab}_x(\Fb_q)|}
        \]
    \end{defn}
    
    It is easy to see that when $K=\Fb_q$, the assignment $[X] \mapsto \#_q(X)$ gives a well-defined ring homomorphism $\#_q: K_0(\mathrm{Stck}_{\Fb_q}) \rightarrow \Q$ (c.f. \cite[\S 2]{Ekedahl}). Henceforth, for any operation on the Grothendieck ring, there is a corresponding identity in the weighted point count. For example, in the setup of Lemma~\ref{lem:Gm_quot}, $\#_q(\mathcal{Y})=\#_q(\mathcal{X})(q-1)^{-1}$, where $\#_q(\Gb_m)=q-1$.
    
    Thus, if we can express $[\Hom_n(\Pb^1,\Pc(a,b))]$ as sums and products of classes of other stacks with known weighted point counts (even if the classes themselves do not decompose into polynomials in $\Lb$), then we can deduce $\#_q(\Hom_n(\Pb^1,\Pc(a,b)))$. Therefore, we will extensively use the Grothendieck ring in the proof of Theorem~\ref{thm:motivecount}, then use Proposition~\ref{Semistable_Bijection} and explicit weighted point counts when necessary. Now we are ready to prove Theorem~\ref{thm:motivecount}.
    
    \subsection{Proof of Theorem~\ref{thm:motivecount}}
    \label{subsec:pfmotivecount}
    
    \par By \cite[\S 5.2]{CCFK}, $\mathrm{Hom}_n(\Pb^1,\Pc(a,b))$ is isomorphic to a stack parameterizing line bundles $\Lc \simeq \varphi_f^*\Oc_{\Pc(a,b)}(1)$ of degree $n$ on $\Pb^1$ together with sections $u \in H^0(\Pb^{1},\Lc^{\otimes a})$ and $v \in H^0(\Pb^{1},\Lc^{\otimes b})$ such that the global sections $u, v$ are not simultaneously vanishing at any points of $\Pb^1$. Moreover, such pairs $(u,v)$ and $(u',v')$ are equivalent when there exists $\lambda \in \Gb_m$ so that $u'=\lambda^au$ and $v'=\lambda^bv$. Consider $T \subset H^0(\Oc_{\Pb^1}(an))\oplus H^0(\Oc_{\Pb^1}(bn))\setminus 0$ a $\Gb_m$-equivariant open subset parameterizing pairs $(u,v)$ with no common zero, where $\Gb_m$-action on the vector space is as above. Then, $\mathrm{Hom}_n(\Pb^1,\Pc(a,b))$ is a smooth stack isomorphic to the quotient stack $[T/\Gb_m]$, admitting $T$ as a smooth schematic cover. In particular, if $\mathrm{char}(K)>0$ and does not divide $a$ or $b$, then $\mathrm{Hom}_n(\Pb^1,\Pc(a,b))$ is a tame stack by \cite[Theorem 3.2]{AOV}. By Lemma~\ref{lem:Gm_quot}, it suffices to obtain the Grothendieck class $[T]$ (or $\#_q(T)$), as 
    \begin{equation}\label{eqn:bdle}
    [\mathrm{Hom}_n(\Pb^1,\Pc(a,b))]=[T][\Gb_m]^{-1}
    \end{equation}
    
    \medskip
    
    \par Now fix a chart $\A^1 \hookrightarrow \Pb^1$ with $x \mapsto [1:x]$, and call $0=[1:0]$ and $\infty=[0:1]$. It comes from a homogeneous chart of $\Pb^1$ by $[Y:X]$ with $x:=X/Y$ away from $\infty$. Then for any $(u,v) \in T$, $u$ and $v$ are homogeneous polynomials in $X$ and $Y$ with degrees $an$ and $bn$ respectively. By plugging in $Y=1$, we obtain representations of $u$ and $v$ as polynomials in $x$ with degrees at most $an$ and $bn$ respectively. For instance, $\deg u < an$ as a polynomial in $x$ if and only if $u(X,Y)$ is divisible by $Y$, i.e. $u$ vanishes at $\infty$. From now on, $\deg P$ means the degree of $P$ as a polynomial in $x$.
    
    \par Denoting $\deg u:=k$ and $\deg v:=l$, then $(u,v) \in T$ is whenever $k=an$ or $l=bn$ (so that they do not simultaneously vanish at $\infty$) and $u,v$ have no common roots. Since there are many possible degrees for a pair $(u,v) \in T$, consider locally closed subsets $T_{k,l}:=\{(u,v) \in T : \deg u=k, \; \deg v=l \}$. Notice that $T_{k-1,bn} \subset \overline{T}_{k,bn}$ as for any $(u,v) \in T_{k-1,bn}$, $u(X,Y)$ has a description as $Y^{an-k+1}u'(X,Y)$ which is $u_{[1:0]}(X,Y)$ from a pencil polynomials $u_{[t_0:t_1]}(X,Y)=Y^{an-k}(t_1Y-t_0X)u'(X,Y)$ where $u_{[1:t_1]} \in T_{k,bn}$. Hence, we obtain the following stratification:
    
    \begin{align*}
    T&=T_{an,bn} \sqcup \left(\bigsqcup_{k=0}^{an-1} T_{k,bn}\right) \sqcup \left(\bigsqcup_{l=0}^{bn-1} T_{an,l} \right)\\
    T&=\overline{T_{an,bn}} \supsetneq \overline{T_{an-1,bn}} \supsetneq \dotsb \supsetneq \overline{T_{0,bn}}=T_{0,bn}\\
    T&=\overline{T_{an,bn}} \supsetneq \overline{T_{an,bn-1}} \supsetneq \dotsb \supsetneq \overline{T_{an,0}}=T_{an,0}
    \end{align*}
    \[\overline{T_{an-k,bn}} \cap \overline{T_{an,bn-l}} =\emptyset ~~\; \forall k,l>0 \]
    Then,
    \begin{equation}
    \label{eqn:sum}
    [T] =[T_{an,bn}]+\sum_{k=0}^{an-1}[T_{k,bn}]+\sum_{l=0}^{bn-1}[T_{an,l}]
    \end{equation}
    
    \par Define 
    \[F_{k,l}:=\{(u,v) \in T_{k,l} :  u,v \text{ are monic} \} \;.\] Then, $F_{k,l} \hookrightarrow T_{k,l}$ is a section of the projection morphism $T_{k,l} \rightarrow F_{k,l}$ (induced by making $(u,v)$ to be a monic pair), which has $\Gb_m \times \Gb_m$--fibers. Hence, $T_{k,l}$ is a $\Gb_m \times \Gb_m$--bundle over $F_{k,l}$, so Lemma~\ref{lem:Gm_quot} implies that
    \begin{equation}\label{eqn:prod}
    [T_{k,l}]=[\Gb_m]^2[F_{k,l}]
    \end{equation}
    
    There is an alternative description of $F_{k,l}$ as below (inspired by \cite{FW}):
    
    \begin{defn}
        Fix a field $K$ with algebraic closure $\overline{K}$. Fix $k,l \ge 0$. Define $\text{Poly}_1^{(k,l)}$ to be the set of pairs $(u,v)$ of monic polynomials in $K[z]$ so that:
        \begin{enumerate}
            \item $\deg u=k$ and $\deg v=l$.
            \item $u$ and $v$ have no common root in $\overline{K}$.
        \end{enumerate}
    \end{defn}
    
    \par Therefore, $F_{k,l} \cong \mathrm{Poly}_1^{(k,l)}$. To finish the proof, it suffices to find descriptions of $[\mathrm{Poly}_1^{(k,l)}]$ and $\#_q(\mathrm{Poly}_1^{(k,l)})$ as polynomials of $\Lb$ and $q$ respectively. Farb and Wolfson \cite{FW} (see \cite{FW2} for corrections to both results and proofs) found such expression when $k=l$ ($\mathrm{Poly}_1^{(k,k)}$ is called $\mathrm{Poly}_1^{k,2}$ in loc. cit.), and we claim that $[\mathrm{Poly}_1^{(k,l)}]$ and $\#_q(\mathrm{Poly}_1^{(k,l)})$ have similar descriptions, as below:
    
    \begin{prop}
        \label{prop:poly}
        Fix $d_1,d_2 \ge 0$. Then, if $\mathrm{char}(K)=0$
        \[ [\mathrm{Poly}_1^{(d_1,d_2)}] =
        \begin{cases}
        \Lb^{d_1+d_2}-\Lb^{d_1+d_2-1}, & \text{ if } d_1,d_2 >0\;, \\
        \Lb^{d_1+d_2}, & \text{ if } d_1=0 \text{ or } d_2=0\;.
        \end{cases}
        \]
        Similarly, for a finite field $\Fb_q$,
        \[ \#_q(\mathrm{Poly}_1^{(d_1,d_2)}) =
        \begin{cases}
        q^{d_1+d_2}-q^{d_1+d_2-1}, & \text{ if } d_1,d_2 >0\;, \\
        q^{d_1+d_2}, & \text{ if } d_1=0 \text{ or } d_2=0\;.
        \end{cases}
        \]
    \end{prop}
    
    \begin{proof}
        The proof for this is analogous to \cite{FW} (see \cite{FW2} for corrections), Theorem 1.2 (1). Here, we only state the differences to their work.
        
        \noindent \textbf{Step 1}: The space of $(u,v)$ monic polynomials of degree $d_1,d_2$ is instead the quotient $\A^{d_1} \times \A^{d_2}/(S_{d_1} \times S_{d_2}) \cong \A^{d_1+d_2}$. We have the same filtration of $\A^{d_1+d_2}$ by $R_{1,k}^{(d_1,d_2)}$, which is the space of $(u,v)$ monic polynomials of degree $d_1,d_2$ respectively for which there exists a monic $h \in K[z]$ with $\deg (h) \ge k$ and monic polynomials $g_i \in K[z]$ so that $u=g_1h$ and $v=g_2h$. The rest of the arguments follow analogously, keeping in mind that the group action is via $S_{d_1} \times S_{d_2}$.

        \noindent \textbf{Step 2}: Fix $k \ge 0$. Consider the morphism \[\overline{\Psi}:\Ab^{(d_1-k)+(d_2-k)} \times \Ab^k \rightarrow \Ab^{d_1+d_2}\]
        where $\overline{\Psi}(f_1,f_2,g)=(f_1g,f_2g)$. This restricts to a morphism \[\Psi: \mathrm{Poly}_1^{(d_1-k,d_2-k)} \times \Ab^k \rightarrow R_{1,k}^{(d_1,d_2)}\setminus R_{1,k+1}^{(d_1,d_2)}\]
        The rest of the arguments follow analogously from \cite{FW2}.

        \noindent \textbf{Step 3}: By combining Step 1 and 2 as in \cite{FW} and \cite{FW2}, if $\mathrm{char}(K)=0$, we obtain
        \[ [\mathrm{Poly}_1^{(d_1,d_2)}]=\Lb^{d_1+d_2}-\sum_{k \ge 1}[\mathrm{Poly}_1^{(d_1-k,d_2-k)}]\Lb^k \]
        
        \par For the induction on the class $[\mathrm{Poly}_1^{(d_1,d_2)}]$, we use lexicographic induction on the pair $(d_1,d_2)$. Since the order of $d_1,d_2$ does not matter for Grothendieck class, we assume that $d_1 \ge d_2$. For the base cases, consider when $d_2=0$. Then the monic polynomial of degree 0 is nowhere vanishing, so that any polynomial of degree $d_1$ constitutes a member of $\mathrm{Poly}_1^{(d_1,0)}$, so that $\mathrm{Poly}_1^{(d_1,0)} \cong \Lb^{d_1}$. Since this argument is independent of the characteristic, $\#_q(\mathrm{Poly}_1^{(d_1,0)})=q^{d_1}$. Similarly, $d_1=0$ is taken care of. Then for $d_1,d_2 >0$, we obtain:
        \begingroup
        \allowdisplaybreaks
        \begin{align*}
        [\mathrm{Poly}_1^{(d_1,d_2)}]&=\Lb^{(d_1+d_2)}-\sum_{k \ge 1} [\mathrm{Poly}_1^{(d_1-k,d_2-k)}]\Lb^k\\
        &=\Lb^{d_1+d_2}-\left(\sum_{k=1}^{d_2-1}(\Lb^{(d_1-k)+(d_2-k)}-\Lb^{(d_1-k)+(d_2-k)-1})\Lb^k+\Lb^{d_1-d_2}\Lb^{d_2}\right)\\
        &=\Lb^{d_1+d_2}-\left(\sum_{k=1}^{d_2-1}(\Lb^{d_1+d_2-k}-\Lb^{d_1+d_2-k-1})+\Lb^{d_1}\right)\\
        &=\Lb^{d_1+d_2}-\Lb^{d_1+d_2-1}
        \end{align*}
        \endgroup
        The same arguments with $\#_q$ in lieu of $[-]$ gives the point count
        \[ \#_q(\mathrm{Poly}_1^{(d_1,d_2)})=q^{d_1+d_2}-q^{d_1+d_2-1} \]
    \end{proof}
    
    \par Applying the Proposition~\ref{prop:poly} to the equations (\ref{eqn:bdle}), (\ref{eqn:sum}) and (\ref{eqn:prod}), we finally get:
    \begingroup
    \allowdisplaybreaks
    \begin{align*}
    &[\mathrm{Hom}_n(\Pb^1,\Pc(a,b))]\\
    &=[\Gb_m]^{-1}[T]\\
    &=[\Gb_m]^{-1}\left([T_{an,bn}]+\sum_{k=0}^{an-1}[T_{k,bn}]+\sum_{l=0}^{bn-1}[T_{an,l}]\right)\\
    &=[\Gb_m]^{-1}[\Gb_m]^2\left([F_{(an,bn)}]+\sum_{k=0}^{an-1}[F_{(k,bn)}]+\sum_{l=0}^{bn-1}[F_{(an,l)}]\right)\\
    &=[\Gb_m]\left([\mathrm{Poly}_1^{(an,bn)}]+\sum_{k=0}^{an-1}[\mathrm{Poly}_1^{(k,bn)}]+\sum_{l=0}^{bn-1}[\mathrm{Poly}_1^{(an,l)}]\right)\\
    &=(\Lb-1)\left((\Lb^{(a+b)n}-\Lb^{(a+b)n-1})+\Lb^{bn}+\sum_{k =1}^{an-1}(\Lb^{bn+k}-\Lb^{bn+k-1})\right)\\
    &\phantom{=}\text{ }+(\Lb-1)\left(\Lb^{an}+\sum_{l=1}^{bn-1}(\Lb^{an+l}-\Lb^{an+l-1})\right)\\
    &=(\Lb-1)(\Lb^{(a+b)n}-\Lb^{(a+b)n-1}+\Lb^{bn}+\Lb^{(a+b)n-1}-\Lb^{bn}\\
    &\phantom{=(\Lb-1)(}+\Lb^{an}+\Lb^{(a+b)n-1}-\Lb^{an})\\
    &=\Lb^{(a+b)n+1}-\Lb^{(a+b)n-1}
    \end{align*}
    \endgroup
    and similarly for $\#_q$, \[ \#_q(\mathrm{Hom}_n(\Pb^1,\Pc(a,b)))=q^{(a+b)n+1}-q^{(a+b)n-1} \; . \]
    
    \par This finishes the proof of Theorem~\ref{thm:motivecount}.
    
    \begin{rmk}\label{rmk:ptcount}
        Fix $n>0$. Since any $\varphi_g \in \mathrm{Hom}_n(\Pb^1,\Pc(a,b))$ is surjective, the generic stabilizer group $\mu_{\gcd(a,b)}$ of $\Pc(a,b)$ is the automorphism group of $\varphi_g$. Then the Definition~\ref{def:wtcount} and Corollary~\ref{cor:pointcount} implies that the number of $\Fb_q$-isomorphism classes of semistable elliptic surfaces of discriminant degree $12n$ over $\Fb_q$ is  
        \[|\Lc_{1, 12n}(\Fb_q)|=2 \cdot (q^{10n+1}-q^{10n-1})\;\] 
        where the factor of 2 comes from the hyperelliptic involution.
    \end{rmk}

    \section{Counting semistable elliptic curves over global fields by $\Delta$}\label{sec:globfield}
    
    \par In this section, we consider $\Zc_{\Fb_q(t)}(\Bc)$ the counting function of semistable elliptic surfaces (Definition~\ref{def:ellsurf}) with $12n$ nodal singular fibers and a distinguished section. We explicitly compute $\Zc_{\Fb_q(t)}(\Bc)$ by the arithmetic invariant $|\Lc_{1, 12n}(\Fb_q)|$ in the function field setting. An analogous object in the number field setting is $\Zc_{\Qb}(\Bc)$ which is the counting of semistable elliptic curves over $\Qb$. In the end, we formulate a heuristic that for both of the global fields the countings $\Zc_{K}(\Bc)$ will match with one another.
    
    \par As the generic point of $\Pb^{1}_{\Fb_q}$ (the base of semistable elliptic fibrations) is indeed Spec of a rational function field of one variable $t$ over $\Fb_q$, one could think of a semistable elliptic surface $X$ over $\Pb^1$ as the choice of a model for semistable elliptic curves $E$ over $K = \Fb_q(t)$ or equivalently over $\Oc_K = \Fb_q[t]$ by clearing the denominators. On the number field, the analogy would be the semistable elliptic curves $E$ with the squarefree conductor $\Nc = p_1 \cdots \cdots p_{\mu}$ over $\Qb$ or equivalently over $\Oc_K = \Zb$ as relative curves over a Dedekind scheme by the minimal integral Weierstrass model of an elliptic curve. In order to draw the analogy, we need to fix an affine chart $\A^1_{\Fb_q} \subset \Pb^1_{\Fb_q}$ and its corresponding ring of functions $\Fb_q[t]$, since $\Fb_q[t]$ could come from any affine chart of $\Pb^1_{\Fb_q}$, whereas the ring of integers for the number field $K$ is canonically determined. We denote $\infty \in \Pb^1_{\Fb_q}$ to be the unique point not in the chosen affine chart.

    \par Note that for a maximal ideal $\pf$ in $\Oc_K$, the residue field $\Oc_K / \pf$ is finite for both of our global fields. One could think of $\pf$ as a point in $\Spec~\Oc_K$ and define the \textit{height} of a point $\pf$.
    
    \begin{defn}
        Define the height of a point $\pf$ to be $ht(\pf) := |\Oc_K / \pf|$ the cardinality of the residue field $\Oc_K / \pf$. 
    \end{defn}

    \par For simplicity, assume that $X$ does not have a singular fiber over $\infty \in \Pb^1_{\Fb_q}$. Note that the primes $\pf$ of bad reductions are precisely points of the discriminant divisor $\Delta$, as the fiber $X_{\pf}$ is singular over $\Delta$. When $K = \Fb_q(t)$ the function field, we have $\Delta(X) \in H^0(\Pb^{1}, \Oc(12n))$. It has the following factorization for pairwise distinct maximal ideals $\pf_i \subset \Fb_q[t]$ and $\alpha \in \Fb^{*}_q$ over the affine chart:
    \[\Delta(X) = -16(4a_4^3+27a_6^2) = \alpha \prod\limits_{i=1}^{\mu} \pf_i^{k_i}\]
    
    \par There are two ways in which the bad reductions can occur: $E$ can become nodal which is called a multiplicative reduction at $\pf$ or $E$ can become cuspidal which is called an additive reduction at $\pf$. For our consideration, we only have multiplicative reductions as possible bad reductions since semistable elliptic fibrations contain only singular fibers of type $I_k$ for $k\geq 1$. Similar to Remark~\ref{rmk:disc}, a given semistable elliptic fibration over the number field $K$ has $12n$ nodal points distributed over $\mu$ distinct singular fibers that are $I_{k_1},~ \cdots ,~I_{k_i},~ \cdots, I_{k_\mu}$ with $\sum\limits_{i=1}^{\mu} k_i = 12n$.
    
    \par As the discriminant divisor $\Delta(X)$ is an invariant of the choice of semistable model $f : X \to \Pb^1$, we count the number of isomorphism classes of nonsingular semistable elliptic fibrations on the function field $\Fb_q(t)$ by the bounded height of $\Delta(X)$:   
    \[ht(\Delta(X)) = \prod\limits_{i=1}^{\mu} |\Fb_q|^{k_i} = q^{k_1} \cdots q^{k_i} \cdots q^{k_\mu} = q^{k_1 + \cdots + k_{\mu}} = q^{12n}\]
    In general, the height of a discriminant $\Delta(X)$ of any $X$ (without nonsingular fiber assumption over $\infty$) is defined as $q^{12n}$ where $Deg(\Delta(X))=12n$. 
    
    \medskip

    \par We now define $\Zc_{\Fb_q(t)}(\Bc)$ and compute it by the arithmetic invariant $|\Lc_{1, 12n}(\Fb_q)|$ which is equivalent to the counting of the semistable elliptic surfaces over $\Fb_q$ by the bounded height of discriminant $\Delta(X)$.
    \[\Zc_{\Fb_q(t)}(\Bc) := |\{\text{Semistable elliptic curves over } \Pb^{1}_{\Fb_q} \text{ with } 0<ht(\Delta(X)) \le \Bc\}|\]
    \begin{thm} [Computation of $\Zc_{\Fb_q(t)}(\Bc)$] 
        The counting of semistable elliptic curves over $\Pb^1_{\Fb_q}$ by $ht(\Delta(X)) = q^{12n} \le \Bc$ satisfies the following inequality:
        \[ \Zc_{\Fb_q(t)}(\Bc) \le 2 \cdot \frac{(q^{11} - q^{9})}{(q^{10}-1)} \cdot \left( \Bc^{\frac{5}{6}} - 1\right)\]
        which is an equality when $\Bc = q^{12n}$ for some $n \in \N$ implying that the acquired upper bound is a sharp asymptotic with the lower order term of zeroth order (i.e. constant).
    \end{thm}

    \begin{proof}
        
        \par Knowing the number of $\Fb_q$-isomorphism classes of semistable elliptic surfaces of discriminant degree $12n$ over $\Fb_q$ is $|\Lc_{1, 12n}(\Fb_q)|= 2 \cdot (q^{10n+1}-q^{10n-1})$ by Remark~\ref{rmk:ptcount}, we can explicitly compute the bounds for $\Zc_{\Fb_q(t)}(\Bc)$ as the following,
        
        \begingroup
        \allowdisplaybreaks
        \begin{equation}
        \begin{split}
        \Zc_{\Fb_q(t)}(\Bc) & = \sum \limits_{n=1}^{\left \lfloor \frac{log_q \Bc}{12} \right \rfloor} |\Lc_{1, 12n}(\Fb_q)| = \sum \limits_{n=1}^{\left \lfloor \frac{log_q \Bc}{12} \right \rfloor} 2 \cdot (q^{10n + 1} - q^{10n - 1}) \\
        & = 2 \cdot ({q^{1} - q^{-1}}) \sum \limits_{n=1}^{\left\lfloor\frac{log_q \Bc}{12}\right\rfloor} q^{10n} \le 2 \cdot ({q^{1} - q^{-1}}) \left( q^{10} + \cdots + q^{10 \cdot (\frac{log_q \Bc}{12})}\right) \\
        & = 2 \cdot ({q^{1} - q^{-1}}) \frac{q^{10} ( \Bc^{\frac{5}{6}} - 1) }{(q^{10}-1)} = 2 \cdot 
        \frac{(q^{11} - q^{9})}{(q^{10}-1)} \cdot ( \Bc^{\frac{5}{6}} - 1)
        \end{split}
        \end{equation}
        \endgroup
        
        On the second line of the equations above, inequality becomes an equality if and only if $n:= \frac{log_q \Bc}{12} \in \N$, i.e. $\Bc=q^{12n}$ for some $n \in \N$. This implies that the acquired upper bound on $\Zc_{\Fb_q(t)}(\Bc)$ is a sharp asymptotic of order $\Oc\left(\Bc^{\frac{5}{6}}\right)$ with the lower order term of zeroth order.
    \end{proof}
    
    \medskip
    
    \par Switching to the number field realm with $K = \Qb$ and $\Oc_K = \Zb$, one could choose the minimal integral Weierstrass model of an elliptic curve with the given discriminant divisor $\Delta$ which is already a number. 
    
    \par In order to match the counting with the function field, we define the $ht(\Delta)$ to be the cardinality of ring of functions on subscheme $\Spec(\Zb/(\Delta))$. This leads to the following analogue of $\Zc_{K}(\Bc)$ over $\Qb$ which is $\Zc_{\Qb}(\Bc)$.
    
    \medskip
    
    \noindent \[\Zc_{\Qb}(\Bc) = |\{\text{Semistable elliptic curves } E \text{ over } \Spec~\Zb \text{ with } 0<ht(\Delta) \le \Bc ~\}|\]
    
    \begin{conj} [Heuristic on $\Zc_{\Qb}(\Bc)$]
        The counting $\Zc_{\Qb}(\Bc)$ of semistable elliptic curves over $\Zb$ by $ht(\Delta) \le \Bc$ follows from the sharp asymptotic counting on $\Zc_{\Fb_q(t)}(\Bc)$ through the global fields analogy. Namely, $\Zc_{\Qb}(\Bc)$ has the leading term of order $\Oc\left(\Bc^{\frac{5}{6}}\right)$ and the lower order term of zeroth order (i.e. constant).
    \end{conj}
    
    \medskip
    
    \par The heuristic estimate of all elliptic curves over $\Qb$ by the bounded height of $\Delta$ was known to have the order of $\Oc\left(\Bc^{\frac{5}{6}}\right)$ by the work of \cite{BM}. Moreover, the counting of stable elliptic curves with squarefree $\Delta$ has been done in the past over $\Qb$ by the work of \cite{Baier} where the leading term has the order of $\Oc\left(\Bc^{\frac{5}{6}}\right)$ and the error term has the order of $\Oc\left(\Bc^{(7-\frac{5}{27}+\epsilon)/12}\right)$. It would be interesting if one could actually show the lower order term of $\Zc_{\Qb}$ for the number of semistable elliptic curves with non-squarefree $\Delta$ over number field $\Qb$ to be of zeroth order as shown here by the sharp asymptotic counting of $\Zc_{\Fb_q(t)}$ over (global) function fields $\Fb_q(t)$ when $\text{char} (\Fb_q) \neq 2,3$.

\end{document}